\documentclass[11pt]{amsart}
\usepackage[T1]{fontenc}
\usepackage{lmodern,microtype}
\usepackage{amsmath,amssymb,amsthm}
\usepackage[a4paper,margin=28mm]{geometry}
\usepackage[colorlinks=true,linkcolor=blue,citecolor=blue,urlcolor=blue]{hyperref}

\newtheorem{theorem}{Theorem}
\newtheorem{lemma}{Lemma}[section]
\newtheorem{corollary}[theorem]{Corollary}
\newtheorem{proposition}[lemma]{Proposition}
\theoremstyle{remark}
\newtheorem{remark}[lemma]{Remark}
\newcommand{\C}{\mathbb C}

\newcommand{\Z}{\mathbb Z}
\newcommand{\cO}{\mathcal O}
\newcommand{\Disc}{\operatorname{Disc}}
\newcommand{\rd}{\operatorname{rd}}

\title[One Arithmetic Gadget, Two Problems]{One Arithmetic Gadget, Two Problems:\\
Small Sum-Product Growth and Many Unit Distances}

\author{J\'ozsef Solymosi}

\address{Department of Mathematics,
University of British Columbia,
Vancouver, BC, Canada}

\address{\'Obuda University,
Budapest, Hungary}

\email{solymosi@math.ubc.ca}

\date{}

\begin{document}

\begin{abstract}
There is an absolute constant $c>0$ such that, for every
$0<\varepsilon\leq1$, arbitrarily large finite sets $A\subset\C$ satisfy
\[
 |A+A|\leq |A|^{1+\varepsilon},\qquad
 |AA|\leq |A|^{2-c\varepsilon},\qquad
 \nu_1(A)\geq |A|^{1+c\varepsilon},
\]
where $\nu_1(A)$ counts unordered unit-distance pairs.
We combine the arithmetic directions from the recent unit-distance
construction with the multiplicative enlargement used in the recent
sum-product construction. The arithmetic ingredients are stated as
explicit inputs.
\end{abstract}

\maketitle

\section{Introduction}

For a finite set $A\subset\C$, write
\[
 A+A=\{a+b:a,b\in A\},\qquad
 AA=\{ab:a,b\in A\},
\]
and let
\[
 \nu_1(A)
 =\#\bigl\{\{a,b\}\subset A:|a-b|=1\bigr\}.
\]
Recent constructions show that finite sets of real numbers can have
simultaneous power savings in their sumsets and product sets
\cite{BSSZ}, and that planar sets of $n$ points can determine at least
$n^{1+\eta}$ unit distances for some fixed $\eta>0$
\cite{ABGLSSTWW}. We show that one set of complex numbers can satisfy
both properties.

\begin{theorem}\label{thm:main}
There is an absolute constant $c>0$ such that, for every
$0<\varepsilon\leq1$, there are arbitrarily large finite sets
$A\subset\C$ satisfying
\[
 |A+A|\leq |A|^{1+\varepsilon},\qquad
 |AA|\leq |A|^{2-c\varepsilon},\qquad
 \nu_1(A)\geq |A|^{1+c\varepsilon}.
\]
\end{theorem}

In particular, this implies the formulation with an unspecified
$\delta(\varepsilon)>0$ for every $\varepsilon>0$: apply the theorem
with $\min\{\varepsilon,1\}$.

\begin{corollary}\label{cor:symmetric}
For some absolute $c_0>0$, arbitrarily large finite sets $A\subset\C$
satisfy
\[
 |A+A|+|AA|\leq |A|^{2-c_0},\qquad
 \nu_1(A)\geq |A|^{1+c_0}.
\]
\end{corollary}

\subsection*{The combinatorial idea}

We identify the plane with $\C$, so a unit vector is a complex
number $u$ with $|u|=1$, and multiplication by $u$ is a rotation.
The construction uses a finite set $U$ of unit vectors and two
finite sets $B_-\subset B\subset\C$ with
\[
 B_-+U\subset B.
\]
Thus every $x\in B_-$ has a unit-distance neighbour $x+u\in B$
for every $u\in U$. We then take disjoint rotated copies of a
translate of $B$:
\[
 A=\bigcup_{u\in U}u(T+B)=U(T+B).
\]
The resulting counts have three useful features:
\begin{itemize}
\item There are $|U|$ disjoint copies, each retaining the
unit-distance pairs already present in $B$.
\item A product of two points of $A$ is determined by a product
from $U^2$ and two points of $B$. A small $U^2$ therefore gives
a saving in $|AA|$.
\item A sum has the form $T(u+v)+(ux+vy)$. We can bound the
number of possibilities for each term independently of the size of $T$.
\end{itemize}

The arithmetic construction supplies $U$ and a family of finite
sets from which we choose $B_-$ and $B$. Section~\ref{sec:input}
states exactly the properties needed, using only finite subsets of
$\C$ and labels in an integer cube. After accepting that input,
the proof in Sections~\ref{sec:directions}--\ref{sec:construction}
uses elementary counting. Appendix~\ref{app:arithmetic} explains
how the input follows from the cited number-theoretic constructions.

\subsection*{Relation to earlier work}

The number-field family and the labelled directions come from
\cite[Section~2, especially the proof of Lemma~2.2]{ABGLSSTWW}.
The lattice counting argument appears in
\cite[proof of Lemma~2.1]{ABGLSSTWW} and
\cite[Lemma~3.3]{BSSZ}. The construction $A=U(T+B)$ follows the
architecture $A=GP$ in \cite[Sections~2 and~4]{BSSZ}.
Their Corollary~1.3 already gives the two sum--product inequalities
over the reals. Our additional conclusion is that these inequalities
and many unit distances can hold for the same complex set.

Our starting observation was that the arithmetic directions also
have slow iterated sum--product growth. We retain that observation,
and the connection with multiplicative rank and the Subspace Theorem,
in Section~\ref{sec:observations}; neither is needed for the main proof.

\section{The arithmetic input in combinatorial form}\label{sec:input}

For finite sets $X,Y\subset\C$, write $XY=\{xy:x\in X,y\in Y\}$.
More generally, $X^{(k)}$ denotes the set of products of $k$ elements
of $X$, and $kX$ the set of sums of $k$ elements; repetitions are
allowed. We also write $X^2=XX$.

The next proposition records the arithmetic construction in a form
that can be used without number-theoretic terminology. It is a
consequence of the cited arguments, rather than a statement appearing
verbatim in either paper. In particular, the bounded multiplicity
in part~\textup{(iii)} is justified in Appendix~\ref{app:arithmetic}.

\begin{proposition}[Input from the arithmetic construction]
\label{prop:input}
There is a fixed constant $C>1$ and an integer $\ell\geq2$
such that, for arbitrarily large positive integers $f$, the following
objects exist simultaneously.

\smallskip
\noindent\textup{(i) Finite sets closed under controlled sums and products.}
For every positive integer $M$ and real $R\geq1$, there is a finite
set $W_M(R)\subset\C$. It contains $0$, is increasing with $R$,
and satisfies, for positive integers $M,N$ and $R,S\geq1$,
\begin{equation}\label{eq:window-operations}
 \begin{split}
 W_M(R)+W_M(S)&\subset W_M(R+S),\\
 W_M(R)W_N(S)&\subset W_{MN}(RS).
 \end{split}
\end{equation}

\smallskip
\noindent\textup{(ii) Size estimates.}
\begin{equation}\label{eq:lattice}
 C^{-f}(MR)^{2f}\leq |W_M(R)|\leq C^f(MR)^{2f}.
\end{equation}

\smallskip
\noindent\textup{(iii) Many labelled unit directions.}
For each positive integer $m$, there is a set
$S_m\subset\{0,\ldots,m\}^f$ and an injective assignment
$\mathbf a\mapsto u_{\mathbf a}$ with image $U_m$ such that
\[
 |S_m|\geq\left(\frac{m+1}{C}\right)^f,\qquad
 U_m\subset W_{\ell^{2m}}(1),\qquad
 |u|=1\ (u\in U_m),\qquad U_m\cap(-U_m)=\varnothing.
\]
For every positive integer $k$ and every $\mathbf q\in\Z^f$,
\begin{equation}\label{eq:label-fibres}
 \left|\left\{
 u_{\mathbf a_1}\cdots u_{\mathbf a_k}:
 \mathbf a_i\in S_m,\quad
 \mathbf a_1+\cdots+\mathbf a_k=\mathbf q
 \right\}\right|\leq C^f.
\end{equation}
The constants $C,\ell$ are independent of $f,m,k,M,R$.
\end{proposition}

The parameter $f$ will tend to infinity and make our sets large.
The integer $m$ will first be fixed to ensure enough directions and
a product saving; $R$ will then control the sumset exponent.
The index $M$ keeps track of the change in the family under products.
All the sets in the proposition are already subsets of the ordinary
complex plane.

The labels in part~(iii) need not fill the cube. What matters is
that there are many labels, and that products associated with one
label sum have bounded multiplicity. This counts \emph{distinct
product values}, not the number of tuples giving a product.
For example, in two coordinates with $m=2$, a sum of two labels
lies in $\{0,\ldots,4\}^2$, which has only $25$ elements.
In general there are $(2m+1)^f$ possible sums of two labels.

\section{Counting products of directions}\label{sec:directions}

We suppress fixed numerical constants. In the proof below, $O(f)$
means a quantity bounded in absolute value by a fixed constant times
$f$. Once the direction parameter $m$ is chosen, these constants
are independent of $R$, $f$, and $\varepsilon$.

\begin{lemma}\label{lem:directions}
For every fixed $H>0$, we can choose $m$ and $D=\ell^{2m}$,
independently of $f$, so that the directions $U=U_m\subset W_D(1)$
satisfy
\[
 |U|\geq e^{Hf},\qquad |U^2|\leq e^{-Hf}|U|^2.
\]
\end{lemma}

\begin{proof}
There are at most $(2m+1)^f$ sums of two labels, each accounting
for at most $C^f$ distinct products. Thus
\[
 |U|\geq\left(\frac{m+1}{C}\right)^f,\qquad
 \frac{|U^2|}{|U|^2}
 \leq\left(\frac{C^3(2m+1)}{(m+1)^2}\right)^f.
\]
As $m\to\infty$, the base in the first bound tends to infinity
and the base in the second tends to zero. Choose $m$ sufficiently
large in terms of $H$.
\end{proof}

\section{The simultaneous construction}\label{sec:construction}

For $R\geq2$, write $B=W_D(R)$ and $B_-=W_D(R-1)$.
The input gives
\[
 B_-+U\subset B.
\]
Moreover, because $(R-1)/R\geq1/2$, its size bounds give a
constant $\kappa>0$, independent of $D$ and $R$, such that
\begin{equation}\label{eq:inner-ratio}
 |B_-|\geq e^{-\kappa f}|B|.
\end{equation}
Choose $H>\kappa+1$ in Lemma~\ref{lem:directions} and fix the
resulting $m,D,U$. From now on only $R$ and $f$ vary.

Choose an integer $T$ so that the sets $u(T+B)$, $u\in U$,
are pairwise disjoint, and put
\[
 A=U(T+B),\qquad n=|A|=|U||B|.
\]
Such a $T$ exists: a collision $u(T+x)=v(T+y)$ with $u\ne v$
forces $T=(vy-ux)/(u-v)$, excluding only finitely many values.
Since $D,m$ are fixed, the input bounds imply
\begin{equation}\label{eq:size-bounds}
 \log n=2f\log R+O(f).
\end{equation}
We now count unit-distance pairs, products, and sums.

\medskip
\noindent\emph{Unit distances.}
Every $(x,v)\in B_-\times U$ gives the pair $\{x,x+v\}$ in $B$.
These unordered pairs are distinct because $U\cap(-U)=\varnothing$.
Each of the $|U|$ disjoint rotated copies of $T+B$ retains them.
Consequently,
\[
 \nu_1(A)\geq |U|^2|B_-|
 \geq e^{(H-\kappa)f}|U||B|
 \geq e^f n.
\]

\medskip
\noindent\emph{Products.}
A product from $A$ is determined by an element of $U^2$ and two
elements of $B$. Hence
\[
 |AA|\leq |U^2||B|^2\leq e^{-Hf}n^2\leq e^{-f}n^2.
\]

\medskip
\noindent\emph{Sums.}
The identity
\[
 u(T+x)+v(T+y)=T(u+v)+(ux+vy)
\]
gives
\[
 \begin{aligned}
 |A+A|&\leq |U+U|\,|UB+UB|\\
 &\leq |W_D(2)|\,|W_{D^2}(2R)|\\
 &\leq e^{O(f)}R^{2f}
 =e^{O(f)}n.
 \end{aligned}
\]
The middle line uses the addition and multiplication rules from
Proposition~\ref{prop:input}. In particular, the bound is independent
of $T$.

\begin{proof}[Proof of Theorem~\ref{thm:main}]
We have established
\begin{equation}\label{eq:uniform-bounds}
 \log n=2f\log R+O(f),\qquad
 |A+A|\leq e^{O(f)}n,\qquad
 |AA|\leq e^{-f}n^2,\qquad
 \nu_1(A)\geq e^f n.
\end{equation}
For $0<\varepsilon\leq1$, choose
\[
 R=\exp(K/\varepsilon),
\]
where $K$ is a sufficiently large fixed constant. Then
$\log n\asymp f/\varepsilon$, uniformly in $\varepsilon$.
Choosing $K$ large enough absorbs the factor $e^{O(f)}$ in the
sumset bound into $n^\varepsilon$. Also
$f\geq c\varepsilon\log n$ for some absolute $c>0$, so
\[
 |A+A|\leq n^{1+\varepsilon},\qquad
 |AA|\leq n^{2-c\varepsilon},\qquad
 \nu_1(A)\geq n^{1+c\varepsilon}.
\]
With $R$ fixed, let $f\to\infty$ through the values supplied by
Proposition~\ref{prop:input}. Then $n\to\infty$.
\end{proof}

\begin{proof}[Proof of Corollary~\ref{cor:symmetric}]
Apply the theorem with $\varepsilon=1/2$ and choose
$0<c_0<\min\{1/2,c/2\}$. For sufficiently large $n$,
$n^{3/2}+n^{2-c/2}\leq n^{2-c_0}$, while the unit-distance
bound is at least $n^{1+c_0}$.
\end{proof}

\section{Additional observations}\label{sec:observations}

\subsection{Slow iterated sum--product growth}

The labelled directions themselves give the same order of exponent
as \cite[Theorem~1.4]{BSSZ}, now on the complex unit circle.
For every fixed integer $k\geq3$, there are arbitrarily large
finite $U\subset\{z\in\C:|z|=1\}$ with
\[
 \max\{|kU|,|U^{(k)}|\}
 \leq |U|^{C_*\log k/\log\log k},
\]
where $C_*$ is an absolute constant.

\begin{proof}
Use $U=U_t$ from Proposition~\ref{prop:input}, where $t$ is a
sufficiently large positive integer. Repeated addition in part~(i)
and the size estimate in part~(ii) give
\[
 |U|\geq\left(\frac{t+1}{C}\right)^f,\qquad
 |kU|\leq |W_{\ell^{2t}}(k)|\leq C^f(k\ell^{2t})^{2f}.
\]
There are at most $(kt+1)^f$ sums of $k$ labels from
$\{0,\ldots,t\}^f$. Part~(iii) therefore gives
\[
 |U^{(k)}|\leq C^f(kt+1)^f.
\]
Consequently, with constants independent of $f,t,k$,
\[
 \log|U|\gg f\log t,\qquad
 \log|kU|\ll f(t+\log k),\qquad
 \log|U^{(k)}|\ll f\log(kt+1).
\]
For large $k$, take $t\asymp\log k$. Dividing the latter two
bounds by the first gives the claimed exponent.
For the remaining bounded values of $k$, fix a sufficiently large
$t$ and enlarge $C_*$. For each fixed $k$, letting $f\to\infty$
gives arbitrarily large $U$.
\end{proof}

These bounds motivated the simultaneous construction. However,
the directions alone have a large two-fold sumset, so we needed a modified argument.
\[
 |U+U|=\frac{|U|(|U|+1)}2.
\]

\subsection{Multiplicative rank and unit distances}

The multiplicative rank of a finite set of nonzero complex numbers
is the torsion-free rank of the group it generates. It counts the
independent multiplicative generators after roots of unity are
ignored. The result discussed in
\cite[Section~3.2]{Subspace} says that for every $\eta>0$ there is
$c_\eta>0$ such that, for all sufficiently large $n$, the number
of unit-distance pairs with directions in a multiplicative group
of rank less than $c_\eta\log n$ is less than $n^{1+\eta}$.
Thus rank $o(\log n)$ gives at most $n^{1+o(1)}$ such pairs, and
a construction with $n^{1+\eta}$ unit-distance pairs requires rank
$\Omega_\eta(\log n)$ for its full direction set.

In the classical lattice construction of Erd\H{o}s, the
multiplicative rank of the unit directions is
$O(\log n/\log\log n)$, whereas the additive group generated by
the points has rank two. See \cite[Section~3]{Erdos1946} for the
construction and \cite[Section~3.2]{Subspace} for the rank analysis.
The selected directions $U$ used here have rank $\Theta(f)$, hence
$\Theta(\log n)$ when $R$ is fixed. The short verification is in
Remark~\ref{rem:rank}. It would be interesting to see the
multiplicative properties of the unit directions in constructions with many unit distances.

\subsection{Further question}

The product saving and the unit-distance gain are both proportional
to $\varepsilon$. What is the best simultaneous tradeoff between
these gains under the constraint $|A+A|\leq|A|^{1+\varepsilon}$?
The present argument establishes positive constants but does not
optimize them. More precise rank estimates for constructions with
many unit distances would also be of interest.

\section*{Acknowledgments and disclosure}

This note arose from the author's efforts to understand the
combinatorial consequences of the constructions in
\cite{ABGLSSTWW,BSSZ}. 
OpenAI's ChatGPT assisted with the development and checking of the
construction, literature searches, and the organization of
this manuscript. The author is responsible for the mathematical claims,
references, and final text. 
The author was supported in part by an NSERC Discovery Grant
and by the National Research, Development and Innovation Office
of Hungary (NKFIH), Grant No.~KKP133819 and Excellence~151341.

\appendix
\section{Justification of the arithmetic input}\label{app:arithmetic}

This appendix connects Proposition~\ref{prop:input} to the source
papers and verifies the bounded multiplicity of products. The main
proof only uses the proposition's stated properties.

\subsection{The fields and finite sets}

A CM field is a totally imaginary quadratic extension of a totally
real number field. Let $K$ be such a field of degree $2f$, and let
$\cO_K$ be its ring of integers. Choose one embedding
$\sigma_j:K\hookrightarrow\C$ from each conjugate pair.
The CM involution $x\mapsto\overline x$ acts as complex
conjugation in every embedding. Put
\[
 \|x\|_\infty=\max_j|\sigma_j(x)|.
\]
The proof of \cite[Theorem~1.1, Section~2]{ABGLSSTWW} supplies
constants $\rho,C>1$, a rational prime $\ell$, and such fields
of arbitrarily large degree with
\begin{equation}\label{eq:arithmetic-input}
 \rd(K)\leq\rho,\qquad h_K\leq C^f,\qquad
 \ell\text{ splits completely in }K.
\end{equation}
Here $\rd(K)=|\Disc K|^{1/(2f)}$ is the root discriminant and
$h_K$ is the class number. The existence of this family and its
class-number bound are used directly from the cited proof.

Fix an embedding $\sigma_0:K\hookrightarrow\C$ and set
\[
 W_M(R)=\sigma_0\bigl(\{x\in M^{-1}\cO_K:
                                  \|x\|_\infty\leq R\}\bigr).
\]
An embedding is injective and preserves addition and multiplication.
The inclusions in Proposition~\ref{prop:input}(i) follow from the
triangle inequality, multiplicativity of the embeddings, and the
fact that $\cO_K$ is a ring. Containment of zero and nesting are
immediate.

The size estimates in part~(ii) follow from the lattice arguments
in \cite[proof of Lemma~2.1]{ABGLSSTWW} and
\cite[Lemma~3.3]{BSSZ}. The full image of $M^{-1}\cO_K$ in
$\C^f$ has covolume $M^{-2f}2^{-f}|\Disc K|^{1/2}$ and is
$M^{-1}$-separated in the maximum norm. Averaging translates of a
polydisc of radius $R/2$, then subtracting one of the points found,
gives the lower bound. Packing disjoint polydiscs of radius
$1/(2M)$ gives the upper bound. Since the root discriminant is
bounded, for $R\geq1$ the resulting estimates have the form
\[
 C^{-f}(MR)^{2f}\leq |W_M(R)|\leq C^f(MR)^{2f}
\]
after enlarging the fixed constant $C$. We use this enlarged
constant also in the class-number bound.

\subsection{Directions and their labels}

Write the conjugate pairs of prime ideals above $\ell$ as
$\mathfrak p_j,\overline{\mathfrak p}_j$, $1\leq j\leq f$.
Apply the construction in \cite[proof of Lemma~2.2]{ABGLSSTWW}
with all exponent parameters equal to $m$. Its class-group
pigeonhole argument gives $S_m\subset\{0,\ldots,m\}^f$, a fixed
$\mathbf b\in S_m$, and elements $u_{\mathbf a}\in K$ such that
\[
 |S_m|\geq\frac{(m+1)^f}{h_K},\qquad
 u_{\mathbf a}\overline{u}_{\mathbf a}=1,\qquad
 u_{\mathbf a}\in\ell^{-2m}\cO_K,
\]
and
\begin{equation}\label{eq:valuations}
 (u_{\mathbf a})=
 \prod_{j=1}^f\mathfrak p_j^{2(a_j-b_j)}
             \overline{\mathfrak p}_j^{-2(a_j-b_j)}.
\end{equation}
Here $(u)$ denotes the principal fractional ideal generated by $u$.
The displayed formula is the output of the cited proof: the ideal
ratio indexed by $\mathbf a,\mathbf b$ has exponents $a_j-b_j$,
and replacing a generator $\alpha$ by
$\alpha/\overline\alpha$ doubles these exponents.

Distinct labels give distinct ideals and hence distinct elements.
Also $U_m\cap(-U_m)=\varnothing$, because $u$ and $-u$ generate
the same ideal. The identity $u\overline u=1$ gives modulus one
under every embedding. After applying $\sigma_0$, these are the
directions in Proposition~\ref{prop:input}(iii).

It remains to check \eqref{eq:label-fibres}; this additional count
is not asserted as a separate result in the cited lemma.
Multiplication adds ideal exponents. For $k$ factors with label
sum $\mathbf q$, the exponent at $\mathfrak p_j$ is
$2(q_j-kb_j)$ and the conjugate exponent is its negative.
Thus all these products generate the same principal fractional
ideal. Fix one product $z_0$. For every other distinct product
$z$ with this ideal, $z/z_0$ is an algebraic unit whose conjugates
all have modulus one. These distinct quotients belong to the set
counted by $W_1(1)$, which has at most $C^f$ elements by part~(ii).
This proves the claimed multiplicity bound and completes the
justification of Proposition~\ref{prop:input}.

\subsection{Rank of the selected directions}

\begin{remark}\label{rem:rank}
Fix $m$ as in Lemma~\ref{lem:directions} and let
$\Gamma=\langle U_m\rangle$. The exponents at
$\mathfrak p_1,\ldots,\mathfrak p_f$ define a homomorphism
$\Gamma\to\Z^f$. Its kernel consists of algebraic units whose
conjugates all have modulus one, and is finite by the bound on
$W_1(1)$. Hence the rank $r$ of $\Gamma$ is the dimension of the
rational span of $2(\mathbf a-\mathbf b)$, $\mathbf a\in S_m$.
In particular, $r\leq f$.

On this span some projection onto $r$ coordinates is injective.
Each selected coordinate takes at most $m+1$ values, so
\[
 e^{Hf}\leq |U_m|=|S_m|\leq(m+1)^r,
 \qquad r\geq\frac{Hf}{\log(m+1)}.
\]
Thus $r=\Theta(f)$. For fixed $R\geq2$, the size bounds and $|U|\geq e^{Hf}$ give
$\log n=\Theta(f)$, as claimed in Section~\ref{sec:observations}.
\end{remark}

\end{document}